\documentclass[11pt]{amsart}

\usepackage{amsmath}
\usepackage{amsfonts}
\usepackage{amssymb}
\usepackage{amscd}
\usepackage{amsthm}
\usepackage{filecontents}
\usepackage{framed}
\usepackage{fullpage}
\usepackage{latexsym}

\usepackage{hyperref}
\hypersetup{colorlinks,linkcolor={red},citecolor={blue},urlcolor={blue}}

\evensidemargin\oddsidemargin

\newtheorem{theorem}{Theorem}[section]

\newtheorem{proposition}[theorem]{Proposition}

\theoremstyle{definition}
\newtheorem{definition}[theorem]{Definition}

\newtheorem{remark}[theorem]{Remark}

\numberwithin{equation}{section}

\newcommand{\CC}{\mathbb C}
\newcommand{\HH}{\mathbb H}

\newcommand{\QQ}{\mathbb Q}
\newcommand{\RR}{\mathbb R}
\newcommand{\ZZ}{\mathbb Z}

\newcommand{\SL}{\mathop{\mathrm {SL}}\nolimits}

\newcommand{\Orth}{\mathop{\null\mathrm {O}}\nolimits}

\newcommand{\latt}[1]{{\langle{#1}\rangle}}

\newcommand{\w}{\operatorname{w}}

\begin{document}

\title[Weyl invariant  Jacobi forms: a new approach]{Weyl invariant  Jacobi forms: a new approach}

\author{Haowu Wang}

\address{Max-Planck-Institut f\"{u}r Mathematik, Vivatsgasse 7, 53111 Bonn, Germany}

\email{haowu.wangmath@gmail.com}


\date{\today}


\begin{abstract}
The weak Jacobi forms of integral weight and integral index associated to an even positive definite lattice form a bigraded algebra. In this paper we prove a criterion for this type of algebra being free. As an application, we give an automorphic proof of K. Wirthm\"{u}ller's theorem which asserts that the bigraded algebra of weak Jacobi forms invariant under the Weyl group is a polynomial algebra for any irreducible root system not of type $E_8$. This approach is also applicable to $E_8$. Even if the algebra of $E_8$ Jacobi forms is known to be non-free, we still derive a new structure result.
\end{abstract}

\maketitle

\section{Introduction}
In 1992,  K. Wirthm\"{u}ller \cite{Wir92} studied the invariant theory of root systems with respect to the actions of the Jacobi group and the Weyl group. More precisely, for any irreducible root system $R$ of rank $l$, he defined the so-called Weyl invariant weak Jacobi forms for $R$. These forms are holomorphic functions in many variables which are modular in the complex upper half-plane $\HH$ and are double quasi-periodic in the lattice variable $\mathfrak{z}\in R\otimes\CC$, and are invariant with respect to the action of the Weyl group $W(R)$ on $\mathfrak{z}$. Weyl invariant Jacobi forms have many applications in Frobenius manifold, Gromov-Witten theory and string theory (see \cite{Ber00a, Ber00b, OP19, Sak17, Sak19, ZGH+18}). All such Jacobi forms form a bigraded algebra over $\CC$, graded by the weight and index. 
K. Wirthm\"{u}ller identified these forms with global sections of a selective reflexive sheaf on a certain abelian variety related to the root system.  By means of the theory of algebraic geometry, he successfully proved that the bigraded algebra is freely generated by $l+1$ basic Jacobi forms over the ring $M_*(\SL_2(\ZZ))$ of usual $\SL_2(\ZZ)$-modular forms when $R$ is not of type $E_8$. The weak Jacobi forms of even weight in the sense of M. Eichler and D. Zagier \cite{EZ85} are exactly the Weyl invariant Jacobi forms of type $A_1$. This is why these invariants introduced by K. Wirthm\"{u}ller are called Weyl invariant Jacobi forms. 

Unlike the classical case of $A_1$, K. Wirthm\"{u}ller did not give an explicit construction of generators and his proof is arcane. Due to this defect and the importance of these Jacobi forms, some mathematicians were still studying this problem in the past 30 years. Now, all types of generators have been constructed explicitly in the literature. The generators of types $A_n$, $B_n$, $G_2$ and $D_4$ can be found in \cite{Ber00a, Ber00b}. The generators of type $E_n$ were given in \cite{Sat98, Sak17, Sak19}. The generators of types $D_n$, $C_n$ and $F_4$ were  constructed in \cite{AG19, Adl20}. Besides, K. Wirthm\"{u}ller's theorem can also be reproved using the pull-back trick for all root systems except $E_6$ and $E_7$. We refer to  \cite{AG19} for a detailed proof of the tower of $D_n$. Since the root systems $E_6$ and $E_7$ are exceptional, the pull-back technique does not work.

In this paper, we introduce a new approach to recover  K. Wirthm\"{u}ller's theorem for all root systems. In \cite{Wan20}, we proved a necessary and sufficient condition for the graded algebra of orthogonal modular forms being free. This condition is based on the differential operators introduced in \cite{AI05} and the existence of a remarkable modular form which vanishes precisely on some mirrors of reflections with multiplicity one. We use a similar strategy to investigate Jacobi forms. We first define a differential operator on Jacobi forms, which can be regarded as the Jacobian of  Jacobi forms (see Proposition \ref{prop:Jacobian}). We then construct some weak Jacobi forms $\Phi_R$ which have special divisors associated to roots of root systems and are anti-invariant under the Weyl group. These forms are constructed as theta blocks, namely the product of Jacobi theta functions divided by a power of Dedekind $\eta$-function (see \cite{GSZ19}). We observe that for each root system not of type $E_8$ the Jacobian of free generators is equal to this particular theta block. Inspired by this observation, we establish a sufficient condition for the bigraded algebra of weak Jacobi forms being free (see Proposition \ref{prop:criterion}). This gives a simple proof of K. Wirthm\"{u}ller's theorem.

This method also works for $E_8$. For the root system $E_8$, we proved in \cite{Wan18} that the space of weak Jacobi forms is not a polynomial algebra and each $E_8$ Jacobi form can be written uniquely as a polynomial in nine algebraically independent holomorphic Jacobi forms introduced by K. Sakai \cite{Sak17} with coefficients which are meromorphic $\SL_2(\ZZ)$ modular forms. We still have the theta block $\Phi_{E_8}$ satisfying the properties mentioned above. But there are no algebraically independent weak Jacobi forms of type $E_8$ whose  Jacobian has the same weight and index as $\Phi_{E_8}$. Fortunately, there are indeed algebraically independent weak $E_8$ Jacobi forms whose Jacobian has the same index as $\Phi_{E_8}$. K. Sakai's Jacobi forms satisfy such property. We conclude from this fact that the coefficients in the unique polynomial expression of a $E_8$ Jacobi form of index $m$ can be represented as the quotients of some $\SL_2(\ZZ)$ modular forms by $g^{m-1}$, where $g$ is a $\SL_2(\ZZ)$ modular form defined as the quotient of the Jacobian of Sakai's forms by $\Phi_{E_8}$ (see Theorem \ref{th:E8}).

The layout of this paper is as follows. In \S \ref{sec:Jacobian} we define the Jacobian of Jacobi forms and establish the sufficient condition.  We introduce K. Wirthm\"{u}ller's theorem in \S \ref{sec:Wirthmuller} and present a new proof in \S \ref{sec:proof}. In \S \ref{sec:E8} we prove the structure result of $E_8$ Jacobi forms.

\section{The Jacobian of Jacobi forms}\label{sec:Jacobian}
Let $L$ be an even positive definite lattice of rank $l$ with bilinear form $\latt{\cdot,\cdot}$ and dual lattice $L^*$. Let $G$ be a subgroup of the integral orthogonal group $\Orth(L)$ of $L$. One defines Jacobi forms of lattice index $L$ and invariant under the group $G$, which are a generalization of classical Jacobi forms introduced by M. Eichler and D. Zagier \cite{EZ85}.

\begin{definition}\label{def:JFs}
Let $k\in \ZZ$ and $t\in \ZZ_{\geq 0}$. If a holomorphic function $\varphi : \HH \times (L \otimes \CC) \rightarrow \CC$ satisfies the following transformation laws
\begin{align*}
&\varphi \left( \frac{a\tau +b}{c\tau + d},\frac{\mathfrak{z}}{c\tau + d} \right) = (c\tau + d)^k \exp\left( t\pi i \frac{c\latt{\mathfrak{z},\mathfrak{z}}}{c \tau + d}\right) \varphi ( \tau, \mathfrak{z} ), \quad \left( \begin{array}{cc}
a & b \\ 
c & d
\end{array} \right)   \in \SL_2(\ZZ),\\
&\varphi (\tau, \mathfrak{z}+ x \tau + y)= \exp\left(-t\pi i [ \latt{x,x}\tau +2\latt{x,\mathfrak{z}} ]\right) \varphi ( \tau, \mathfrak{z} ), \quad x,y\in L,
\end{align*}
and if its Fourier expansion takes the form 
\begin{equation*}
\varphi ( \tau, \mathfrak{z} )= \sum_{ n=0}^{\infty}\sum_{ \ell \in L^*}f(n,\ell)e^{2\pi i (n\tau + \latt{\ell,\mathfrak{z}})},
\end{equation*}
and if it is invariant with respect to the action of $G$ on the lattice variable
$$
\varphi(\tau,\sigma (\mathfrak{z}))=\varphi(\tau, \mathfrak{z}), \quad \sigma\in G,
$$
then $\varphi$ is called a $G$-invariant weak Jacobi form of weight $k$ and index $t$ associated to $L$. If $f(n,\ell) = 0$ whenever $2nt - \latt{\ell,\ell} <0$, then $\varphi$ is called a $G$-invariant holomorphic Jacobi form.  
We denote by $J^{\w ,G}_{k,L,t}$ and  $ J^{G}_{k,L,t} $ the vector spaces of $G$-invariant weak and holomorphic  Jacobi forms of weight $k$ and index $t$, respectively.
\end{definition}

We remark that the  Jacobi form of index $0$ is independent of the lattice variable $\mathfrak{z}$ and its definition reduces to that of a classical modular form on $\SL_2(\ZZ)$. Thus $J_{k,L,0}^{\w,G}=M_k(\SL_2(\ZZ))$.

All $G$-invariant weak Jacobi forms associated to $L$ form a bigraded algebra over $\CC$
$$
J_{*,L,*}^{\w,G}:=\bigoplus_{k\in \ZZ, t\in \ZZ_{\geq 0}} J_{k,L,t}^{\w,G}.
$$
We will investigate the algebraic structure of this type of  algebra.

We first define the Jacobian of Jacobi forms which is an analogue of the Jacobian of Siegel modular forms defined in \cite{AI05}. This operator plays a vital role in this paper.

\begin{proposition}\label{prop:Jacobian}
Let $l$ be the rank of $L$. For $1\leq j \leq l+1$, let $\phi_j$ be a $G$-invariant weak Jacobi form of weight $k_j$ and index $m_j$ associated to $L$. We fix a coordinate $\mathfrak{z}=(z_1,...,z_l)$ of the space $L\otimes \CC$.  We define the Jacobian of the $l+1$ Jacobi forms as follows
$$
J:=J(\phi_1,...,\phi_{l+1})=\left\lvert \begin{array}{cccc}
m_1\phi_1 & m_2\phi_2 & \cdots & m_{l+1}\phi_{l+1} \\ 
\frac{\partial \phi_1}{\partial z_1} & \frac{\partial \phi_2}{\partial z_1} & \cdots & \frac{\partial \phi_{l+1}}{\partial z_1} \\ 
\vdots & \vdots & \ddots & \vdots \\ 
\frac{\partial \phi_1}{\partial z_l} & \frac{\partial \phi_2}{\partial z_l} & \cdots & \frac{\partial \phi_{l+1}}{\partial z_l}
\end{array}   \right\rvert.
$$
\begin{enumerate}
\item The function $J$ is a weak Jacobi form of weight $l+\sum_{j=1}^{l+1}k_j$ and index $\sum_{j=1}^{l+1} m_j$ associated to $L$. Moreover, it is invariant under $G$ up to the determinant character $\det$.
\item The function $J$ is not identically zero if and only if the $l+1$ Jacobi forms $\phi_j$ are algebraically independent over $M_*(\SL_2(\ZZ))$.
\item Let $r$ be a primitive vector of $L$. If the reflection 
$$
\sigma_r: \quad L\otimes\CC \to L\otimes\CC, \quad v\mapsto v- \frac{2\latt{r,v}}{\latt{r,r}}r
$$ 
belongs to $G$, then $J$ vanishes on the set 
$$
D_{r}(\tau):=\{ (\tau,\mathfrak{z}) \in \HH \times (L\otimes\CC) : \latt{r^\vee, \mathfrak{z}}\in \ZZ\tau \oplus \ZZ\},
$$
where $r^\vee=2r/(r,r)$ is the coroot of $r$. 
\end{enumerate}
\end{proposition}

\begin{proof}
\begin{enumerate}
\item[(1)] The proof is similar to that of \cite[Proposition 2.1]{AI05}. For $2\leq j \leq l+1$, the function $\varphi_j=\phi_j^{m_1}/\phi_1^{m_j}$ is a meromorphic Jacobi form of weight $k_jm_1-k_1m_j$ and index $0$ invariant under $G$ associated to $L$ (i.e. a meromorphic function on $\HH \times (L \otimes \CC)$ satisfying the transformation laws in Definition \ref{def:JFs}). We define the usual Jacobian of the $l$ forms $\varphi_j$ with respect to $\mathfrak{z}$:
$$
\Psi (\tau,\mathfrak{z}) = \frac{\partial (\varphi_2,...,\varphi
_{l+1}) }{\partial (z_1,...,z_l)}.
$$
We find that $\Psi$ is a meromorphic Jacobi form of weight $l+\sum_{j=2}^{l+1}(k_jm_1-k_1m_j)$ and index $0$ associated to $L$ by verifying the two transformation laws in Definition \ref{def:JFs}. It is easy to check the following:
$$
\frac{\partial}{\partial z_i}\left( \frac{\phi_j^{m_1}}{\phi_1^{m_j}} \right)= \frac{m_1\phi_j^{m_1-1}}{\phi_1^{m_j}} \left[ \frac{\partial \phi_j}{\partial z_i} - \frac{m_j\phi_j}{m_1\phi_1}\times \frac{\partial \phi_1}{\partial z_i}
\right], \quad 1\leq i \leq l.
$$
This implies that
$$
J(\phi_1,...,\phi_{l+1})=\frac{\phi_1^{1+\sum_{j=2}^{l+1}m_j}}{m_1^{l-1}(\phi_2\cdots \phi_{l+1})^{m_1-1}} \Psi.
$$
It follows that $J$ is a weak Jacobi form of weight $l+\sum_{j=1}^{l+1}k_j$ and index $\sum_{j=1}^{l+1}m_j$ associated to $L$. It is obvious that $\Psi(\tau,\sigma(\mathfrak{z}))=\det(\sigma)\Psi(\tau,\mathfrak{z})$ for any $\sigma\in G$. Therefore $J$ is invariant under $G$ up to the determinant character.
\item[(2)] $(\Leftarrow)$ If the $l+1$ Jacobi forms $\phi_j$ are algebraically independent over $M_*(\SL_2(\ZZ))$, then the $l$ functions $\varphi_j$ are local parameters of the $l$-dimensional variety $L\otimes\CC/ (L\cdot \tau \oplus L)$. Thus the usual Jacobian $\Psi$ does not vanish identically, which yields that $J$ is not identically zero. 
\smallskip

\noindent
$(\Rightarrow )$ Suppose that $J\neq 0$. If the $l+1$ modular forms $\phi_j$ are algebraically dependent over $M_*(\SL_2(\ZZ))$, then there exists a non-zero polynomial $P$ over $M_*(\SL_2(\ZZ))$ in $l+1$ variables such that $P(\phi_1,...,\phi_{l+1})=0$. We write 
$$
P(X_1,...,X_{l+1})=\sum_{(i_1,...,i_{l+1})\in \ZZ_{\geq 0}^{l+1}} c(i_1,...,i_{l+1}) X_1^{i_1}\cdots X_{l+1}^{i_{l+1}}.
$$
We can assume that $\sum_{j=1}^{l+1} m_j i_j$ is a fixed constant $c$ for any $(i_1,...,i_{l+1})\in \ZZ_{\geq 0}^{l+1}$ due to the second transformation law in Definition \ref{def:JFs}. We note that $c$ is the index of the Jacobi form $P(\phi_1,...,\phi_{l+1})$.  By taking the differentials of $P(\phi_1,...,\phi_{l+1})$ with respect to $z_1$,...,$z_n$ respectively, we obtain the following system of linear equations 
$$
\mathbf{J} \left(\frac{\partial P}{\partial \phi_1}, \frac{\partial P}{\partial \phi_2},...,\frac{\partial P}{\partial \phi_{l+1}}  \right)^t = \left( cP, \frac{\partial P}{\partial z_1},...,\frac{\partial P}{\partial z_l} \right)^t=0,
$$
where $\mathbf{J}$ is the Jacobian matrix in the definition of $J$. This leads to a contradiction because we can assume that $P$ has the minimal degree. Therefore these $\phi_j$ are algebraically independent over $M_*(\SL_2(\ZZ))$ if $J\neq 0$.
\item[(3)] Let $\sigma_r\in G$ for a primitive $r\in L$. Then we have 
\begin{equation}\label{eq:1}
J(\tau,\sigma_r(\mathfrak{z}))=-J(\tau,\mathfrak{z}).
\end{equation}
The second transformation law in Definition \ref{def:JFs} gives
\begin{equation}\label{eq:2}
J(\tau,\mathfrak{z}+x\tau+y)=\exp(-\pi i m (\latt{x,x}\tau+2\latt{x,\mathfrak{z}}))J(\tau,\mathfrak{z}), \quad x,y\in L,
\end{equation}
where $m$ is the index of $J$ as a Jacobi form. Let $a$ and $b$ be two integers. For any vector $\mathfrak{z}\in L\otimes\CC$ satisfying $\latt{r^\vee,\mathfrak{z}}=a\tau+b$, we can write $\mathfrak{z}=u\tau+v$ with $u,v\in L\otimes\RR$. Then we have that $\latt{r^\vee,u}=a$ and $\latt{r^\vee,v}=b$. 
It is easy to check that $\sigma_r(u\tau+v)=u\tau+v-ar\tau-br$. By taking $\mathfrak{z}=u\tau+v$, $x=-ar$ and $y=-br$ in \eqref{eq:1} and \eqref{eq:2}, we get 
$$
J(\tau,u\tau+v)=-J(\tau,u\tau+v),
$$
which implies that $J$ vanishes when $\mathfrak{z}= u\tau+v$. This finishes the proof of (3).
\end{enumerate}
\end{proof}

In \cite[Theorem 5.1]{Wan20}, we prove a sufficient condition for the graded algebra of orthogonal modular forms to be free. We here prove a similar criterion for Jacobi forms, which provides a simple method to determine the structure of a bigraded algebra of Jacobi forms. 
\begin{proposition}\label{prop:criterion}
Let $L$ be an even positive definite lattice of rank $l$ and $G$ be a subgroup of $\Orth(L)$.  Let $\phi_j$ be a $G$-invariant weak Jacobi form of weight $k_j$ and index $m_j$ associated to $L$ for $1\leq j\leq l+1$. Assume that the Jacobian $J(\phi_1,...,\phi_{l+1})$ is not identically zero and we denote its weight and index by $k_0$ and $m_0$ respectively. We further assume that there is a weak Jacobi form $\hat{J}$ of weight $\hat{k}_0$ and index $m_0$ associated to $L$ which vanishes precisely with multiplicity one on $D_r(\tau)$ for all primitive $r\in L$ satisfying $\sigma_r\in G$, and it has a non-zero Fourier coefficient of type $f(0,\ell)$.
\begin{enumerate}
\item If $k_0=\hat{k}_0$, then the bigraded algebra $J_{*,L,*}^{\w, G}$ is freely generated by the $l+1$ forms $\phi_j$ over $M_*(\SL_2(\ZZ))$.
\item If $k_0>\hat{k}_0$, then for any $\varphi_m\in J_{k,L,m}^{\w,G}$ there exists a unique polynomial $P$ over $M_*(\SL_2(\ZZ))$ in $l+1$ variables such that
$$
g^{m-M+1} \varphi_m = P(\phi_1,...,\phi_{l+1}),
$$
where 
$$
g=J(\phi_1,...,\phi_j) / \hat{J} \in M_{k_0-\hat{k}_0}(\SL_2(\ZZ)),
$$
and $M$ is the minimal index of $G$-invariant weak Jacobi forms not contained in  the space $M_*(\SL_2(\ZZ))[\phi_j,1\leq j\leq l+1]$. We further define meromorphic Jacobi forms $\hat{\phi}_j=\phi_j / g^{m_j}$ for $1\leq j\leq l+1$. Then we have
$$
J_{*,L,*}^{\w,G}\subsetneq M_*(\SL_2(\ZZ))[\hat{\phi}_j, 1\leq j\leq l+1].
$$
\end{enumerate}
\end{proposition}

\begin{proof}
Let $\phi_{l+2}\in J_{k_{l+2},L,m_{l+2}}^{\w,G}$.  For $1\leq t \leq l+2$ we define $J_t$ as the Jacobian of the $l+1$ Jacobi forms $\phi_j$ except $\phi_t$. It is clear that $J(\phi_1,...,\phi_{l+1})=J_{l+2}$.  By Proposition \ref{prop:Jacobian}, the quotient $g_t:=J_t/\hat{J}$ is $G$-invariant and holomorphic on $\HH\times (L\otimes\CC)$.  Since the $q^0$-term of $\hat{J}$ is not zero, the Fourier expansion of $g_t$ has no terms $c(n,\ell)$ with negative $n$. Therefore $g_t$ is a $G$-invariant weak Jacobi form associated to $L$.  It is easy to check that the following identity holds:
$$
\sum_{t=1}^{l+2} (-1)^t m_t \phi_t J_t = 0.
$$
By $J_t=\hat{J}g_t$, we have 
$$
\sum_{t=1}^{l+2} (-1)^t m_t \phi_t g_t = 0,
$$
which yields 
\begin{equation}\label{eq:main}
(-1)^{l+2}m_{l+2}\phi_{l+2} g =-\sum_{t=1}^{l+1}(-1)^t m_t \phi_t g_t
\end{equation}
where $g:=g_{l+2}$. Since $J_{l+2}$ and $\hat{J}$ has the same index, the function $g$ is a weak Jacobi form of index $0$ associated to $L$. Thus $g$ is independent of the variable $\mathfrak{z}$ and is a modular form of weight $k_0-\hat{k}_0$ on $\SL_2(\ZZ)$.
\begin{enumerate}
\item[(i)] When $k_0=\hat{k}_0$, the modular form $g$ has weight $0$ and then it is a non-zero constant. We observe that the index of non-zero $g_t$ is less than the index of $\phi_{l+2}$. By induction on index, we prove the assertion (1) with the help of \eqref{eq:main}.
\item[(ii)] When $k_0>\hat{k}_0$, we suppose that $J_{*,L,*}^{\w,G}$ is not generated by these $\phi_j$, otherwise there is nothing to prove. We assume that $\phi_{l+2}$ is a weak Jacobi form of index $M$ not contained in $M_*(\SL_2(\ZZ))[\phi_j,1\leq j\leq l+1]$.  Again, we note that the index of non-zero $g_t$ is less than the index of $\phi_{l+2}$. By \eqref{eq:main} and the minimality of $M$, we assert that $g\phi_{l+1}\in M_*(\SL_2(\ZZ))[\phi_j,1\leq j\leq l+1]$. We then complete the proof of assertion (2) by induction on index.
\end{enumerate}
\end{proof}

\section{Weyl invariant Jacobi forms}\label{sec:Wirthmuller}
In this section we recall K. Wirthm\"{u}ller's theorem. Let $R$ be an irreducible root system of rank $l$. The classification of $R$ is as follows  (see \cite{Bou60})
\begin{align*}
&A_l (l\geq 1),& &B_l (l\geq 2),& &C_l (l\geq 3),& &D_l (l\geq 4),& &E_6,& &E_7,& &E_8,& &G_2,& &F_4.&
\end{align*}
Let $W(R)$ be the Weyl group of $R$. The roots of $R$ span an integral lattice with the standard bilinear form $(\cdot,\cdot)$ on $\RR^l$. If this lattice is odd, we rescale its bilinear form by $2$ such that it becomes an even lattice. We denote this even positive definite lattice by $L_R$ and its normalized bilinear form by $\latt{\cdot, \cdot}$.  Following Definition \ref{def:JFs}, a Weyl invariant weak Jacobi form of type $R$ is just a $W(R)$-invariant weak Jacobi form associated to $L_R$.

We introduce some notations of root systems.  The dual root system of $R$ is defined as 
\begin{equation*}
R^\vee=\{ r^\vee: r\in R \},
\end{equation*}
where $r^\vee=\frac{2}{(r,r)}r$ is the coroot of $r$.
The weight lattice of $R$ is defined as
$$
\Lambda(R)=\{ x\in R\otimes\QQ: (x,r^\vee)\in \ZZ \}.
$$
Let $\widetilde{\alpha}$ denote the highest root of $R^\vee$.  The following significant theorem was proved by  K. Wirthm\"{u}ller  in 1992.

\begin{theorem}[Theorem 3.6 in \cite{Wir92}]\label{th:Wir}
If $R$ is an irreducible root system of rank $l$ and not of type $E_8$, then $J^{\w ,W(R)}_{*,L_R,*}$ over $M_*(\SL_2(\ZZ))$ is freely generated by $l+1$ $W(R)$-invariant weak Jacobi forms of weight $-k_j$ and index $m_j$
\begin{equation*}
\phi_{-k_j,R,m_j}(\tau,\mathfrak{z}), \quad 1\leq j\leq l+1.
\end{equation*}
Apart from $(k_1, m_1)=(0, 1)$, the indices $m_j$ are the coefficients of $\widetilde{\alpha}^\vee$  written as a linear combination of the simple roots of $R$. The integers $k_j$ are the degrees of the generators of the ring of $W(R)$-invariant polynomials, namely the exponents of the Weyl group $W(R)$ increased by $1$.
\end{theorem}

\begin{table}[ht]
\caption{Weights and indices of generators of Weyl invariant weak Jacobi forms ($B_l: l\geq 2$, $C_l: l\geq 3$, $D_l: l\geq 4$)}
\renewcommand\arraystretch{1.2}
\noindent\[
\begin{array}{c|c|c|c}
R & L_R & W(R) & (k_j,m_j) \\ 
\hline 
A_l & A_l & W(A_l) & (0,1), (s,1) : 2\leq s\leq l+1\\ 
\hline 
B_l & lA_1 & \Orth(lA_1) & (2s,1) : 0\leq s \leq l  \\ 
\hline 
C_l & D_l & W(C_l) & (0,1), (2,1), (4,1), (2s,2) : 3\leq s \leq l  \\ 
\hline 
D_l & D_l & W(D_l) &  (0,1), (2,1), (4,1), (n,1), (2s,2) : 3\leq s \leq l-1 \\ 
\hline
E_6 & E_6 & W(E_6) & (0,1), (2,1), (5,1), (6,2), (8,2), (9,2), (12,
3)  \\ 
\hline
E_7 & E_7 & W(E_7) & (0,1), (2,1), (6,2), (8,2), (10,2), (12,
3), (14,3), (18,4)  \\ 
\hline
G_2 & A_2 & \Orth(A_2) & (0,1), (2,1), (6,2)  \\ 
\hline
F_4 & D_4 & \Orth(D_4) &  (0,1), (2,1), (6,2), (8,2), (12,3) 
\end{array} 
\]
\end{table}

\section{The proof of Wirthm\"{u}ller's theorem}\label{sec:proof}
In this section we present a simple proof of K. Wirthm\"{u}ller's theorem using Proposition \ref{prop:criterion} and some known constructions of generators.

Let $R$ be an irreducible root system. We first construct the particular Jacobi form $\hat{J}$ appearing in Proposition \ref{prop:criterion} for $R$. Recall that the Jacobi theta function
$$
\vartheta(\tau,z)=q^{\frac{1}{8}}(\zeta^{\frac{1}{2}}-\zeta^{-\frac{1}{2}})\prod_{n=1}^\infty (1-q^n\zeta)(1-q^n\zeta^{-1})(1-q^n), \quad q=e^{2\pi i\tau}, \zeta=e^{2\pi i z}
$$
is a holomorphic Jacobi form of weight $1/2$ and index $1/2$ for $A_1$ with a multiplier system of order 8 (see e.g. \cite[\S 1.5]{Gri18}). This function vanishes exactly on $\{ (\tau,z)\in \HH\times\CC: z\in \ZZ\tau \oplus \ZZ \}$ with multiplicity one.
The Dedekind $\eta$-function is defined by
$$
\eta(\tau)=q^{\frac{1}{24}}\prod_{n=1}^\infty (1-q^n).
$$
In terms of the two types of functions,  we define a theta function related to $R$ as follows
$$
\Phi_R(\tau,\mathfrak{z})=\prod_{\substack{r\in R^\vee \\ r>0}} \frac{\vartheta(\tau,(r,\mathfrak{z}))}{\eta^3(\tau)}, \quad \mathfrak{z}\in R\otimes \CC,
$$
where the product takes over all positive roots of $R^\vee$.
The Coxeter number $h^\vee$ of the dual root system $R^\vee$ is defined by the equality
$$
h^\vee(\mathfrak{z},\mathfrak{z})=\sum_{\substack{r\in R^\vee \\ r>0}} (r,\mathfrak{z})^2.
$$
By the above identity, the lattice $\Lambda(R)(h^\vee)$, which is obtained by rescaling the weight lattice of $R$ with $h^\vee$, is integral. The function $\Phi_R$ is a weak Jacobi form of weight $-\lvert R\rvert /2$ and index 1 for $\Lambda(R)(h^\vee)$, where $\lvert R\rvert$ is the number of roots of $R$. Let $\overline{R}$ stand for the lattice generated by roots of $R$. It is easy to check that $\overline{R}\subset \Lambda(R)$. We therefore conclude that
\begin{align*}
&\Phi_R\in J_{-\lvert R \rvert /2, L_R, h^\vee}^{\w}, \quad \text{if $\overline{R}$ is even},\\
&\Phi_R\in  J_{-\lvert R \rvert /2, L_R, h^\vee/2}^{\w}, \quad \text{if $\overline{R}$ is odd}.
\end{align*}

\begin{remark}
These functions $\Phi_R$ are the Kac-Weyl denominator functions of affine Lie algebras and also the infinite products occurring in Macdonald identities up to certain powers of Dedekind $\eta$-function. The automorphic properties of $\Phi_R$ are clear in the literature. We refer to \cite[Theorem 6.5]{Bor95} for a direct proof given by R. Borcherds. In \cite[Corollary 2.7 and formula (20)]{Gri18}, V. Gritsenko gave another proof based on the automorphic properties of $\vartheta$. We remark that V. Gritsenko, N. P. Skoruppa and D. Zagier discovered a new arithmetic proof of Macdonald identities in \cite{GSZ19}.
\end{remark}

We formulate the data of root systems in Tables \ref{table-even} and \ref{table-odd}. 
\begin{table}[h]\caption{even root systems}\label{table-even}
 \[
\renewcommand{\arraystretch}{1.2}
\begin{array}{c|c|c|c|c|c|c}
R & \overline{R} & L_R & R^\vee & \Lambda(R) & h^\vee & \frac{1}{2}\lvert R \rvert \\ \hline
A_l & A_l& A_l & A_l & A_l^* & l+1 & l(l+1)/2 \\ \hline
D_l & D_l& D_l & D_l & D_l^* & 2(l-1) & l(l-1) \\ \hline
E_6 & E_6 & E_6 & E_6 & E_6^* & 12 & 36 \\ \hline
E_7 & E_7 & E_7 & E_7 & E_7^* & 18  & 63 \\  \hline
E_8 & E_8 & E_8 & E_8 & E_8 & 30 & 120 \\ \hline
C_l & D_l& D_l & B_l & \ZZ^l & 2l-1 & l^2 \\ \hline
G_2 & A_2& A_2 & G_2(\frac{1}{3}) & A_2 & 4 & 6
\end{array}
\]
\end{table}
\begin{table}[h]\caption{odd root systems}\label{table-odd}
 \[
\renewcommand{\arraystretch}{1.2}
\begin{array}{c|c|c|c|c|c|c}
R & \overline{R} & L_R & R^\vee & \Lambda(R) & h^\vee & \frac{1}{2}\lvert R \rvert \\ \hline
B_l & \ZZ^l & lA_1 & C_l & D_l^* & 2l+2 & l^2\\ \hline
F_4 & D_4^* & D_4 & F_4(2) & D_4^* & 18 & 24
\end{array}
\]
\end{table}

We have the following claims which can be verified directly.
\begin{itemize}
\item[(a)] The function $\Phi_R(\tau, \mathfrak{z})$ vanishes precisely on $D_r(\tau)$ with multiplicity one for all roots $r$ of $R$.
\item[(b)] The Jacobi form $\Phi_R$ has a non-zero Fourier coefficient of type $f(0,\ell)$.
\item[(c)] When $R\neq E_8$, the sum of indices of all generators of $W(R)$-invariant weak Jacobi forms in Theorem \ref{th:Wir} equals the index of $\Phi_R$.
\item[(d)] When $R\neq E_8$, the sum of weights of all generators of $W(R)$-invariant weak Jacobi forms in Theorem \ref{th:Wir} equals $-\frac{1}{2}\lvert R\rvert-l$.
\end{itemize}

From the claims above, we see that for any irreducible root system $R$ not of type $E_8$ the Jacobian of generators and $\Phi_R$ have the same weight and index.  Thus, if we have some basic $W(R)$-invariant weak Jacobi forms with weights and indices as in Theorem \ref{th:Wir} and if we can show that their Jacobian does not vanish identically or equivalently that they are algebraically independent over $M_*(\SL_2(\ZZ))$, then we deduce from Proposition \ref{prop:criterion} (1) that the algebra $J_{*,L_R,*}^{\w, W(R)}$ is freely generated by these basic Jacobi forms over $M_*(\SL_2(\ZZ))$. This proves the K. Wirthm\"{u}ller theorem.

We explain the existence of these expected algebraic independent $W(R)$-invariant weak Jacobi forms. They have been constructed in the literature as mentioned in the introduction. In the simplest case of $A_1$, there are two unique (up to scale) index one weak Jacobi forms of weights $-2$ and $0$. A direct calculation shows that their Jacobian is not identically zero and thus equal to $\Phi_{A_1}=\vartheta(\tau,2z)/\eta^{3}(\tau)$ up to a constant. 

We next discuss the two most complicated cases namely $E_6$ and $E_7$. For these two root systems, it is very hard to calculate the Fourier expansion of the Jacobian. Hence it is better to prove directly that the constructed Jacobi forms are algebraically independent over $M_*(\SL_2(\ZZ))$. The generators of types $E_6$ and $E_7$ were constructed by K. Sakai in \cite{Sak19}. In the case of $E_7$,  K. Sakai first constructed eight holomorphic Jacobi forms 
$$
A_m^{E_7}\; (m=1,2,3), \quad B_m^{E_7}\; (m=2,3,4), \quad C_m^{E_7}\; (m=1,2)
$$
of weights 4, 6, 6 and index $m$ respectively. He then constructed eight weak Jacobi forms of weights and indices coinciding with Theorem \ref{th:Wir} as quotients of some polynomials over $M_*(\SL_2(\ZZ))$ in the above eight holomorphic Jacobi forms by certain powers of $\Delta=\eta^{24}$ and the Eisenstein series of weight $4$ (see \cite[pages 69--70]{Sak19}). Like the case of $E_8$ studied in \cite[Lemma 3.3]{Wan18}, for each $n$ the $q^n$-term of the Fourier expansion of a $W(E_7)$-invariant  Jacobi form can be expressed as a polynomial over $\CC$ in seven Weyl orbits of fundamental weights of $E_7$ due to the Weyl invariance. These Weyl orbits appear in $q^0$-terms of Fourier expansions of Sakai's weak Jacobi forms. We can further construct eight weak Jacobi forms as polynomials over $M_*(\SL_2(\ZZ))$ in Sakai's weak Jacobi forms whose $q^0$-terms of Fourier expansions contain only the constant one and the seven Weyl orbits respectively. From the algebraic independence of the seven Weyl orbits over $\CC$, we deduce the algebraic independence of the eight forms over $M_*(\SL_2(\ZZ))$, which forces that Sakai's weak Jacobi forms are also algebraically independent over $M_*(\SL_2(\ZZ))$. We then recover the Wirthm\"{u}ller theorem for root system $E_7$. The case of $E_6$ is similar.

\begin{remark}
One can define the bigraded algebra $J_{*,L_R,*}^{\w,W(R)}(\Gamma)$ of Weyl invariant weak Jacobi forms with respect to a congruence subgroup $\Gamma$ of $\SL_2(\ZZ)$.  In fact,  K. Wirthm\"{u}ller proved the structure result for any $\Gamma$ in \cite[Theorem 3.6]{Wir92}. More precisely, he showed that $J_{*,L_R,*}^{\w,W(R)}(\Gamma)$ is freely generated by the $r+1$ basic forms $\phi_{-k_j,R,m_j}$ of Theorem \ref{th:Wir} over the ring $M_*(\Gamma)$ of elliptic modular forms on $\Gamma$ when $R\neq E_8$. We note that these $\phi_{-k_j,R,m_j}$ are independent of $\Gamma$. Clearly, our proof holds for any $J_{*,L_R,*}^{\w,W(R)}(\Gamma)$. Thus it also gives a proof of Wirthm\"{u}ller's theorem for arbitrary $\Gamma$.
\end{remark}

\section{Weyl invariant \texorpdfstring{$E_8$}{E8} Jacobi forms}\label{sec:E8}
In this section we study Weyl invariant $E_8$ Jacobi forms. K. Sakai \cite{Sak17,Sak19} constructed nine $W(E_8)$-invariant holomorphic Jacobi forms
$$
A_m^{E_8}\; (m=1,2,3,4,5), \quad B_m^{E_8}\; (m=2,3,4,6)
$$
of weights 4, 6 and index $m$ respectively. We showed in the proof of \cite[Theorem 4.1]{Wan18} that the nine Jacobi forms are algebraically independent over $M_*(\SL_2(\ZZ))$. The sum of the indices of the nine forms equals 30 which is also the index of $\Phi_{E_8}$. Applying Proposition \ref{prop:criterion} to this case, we find $M=2$ and prove the following theorem which is stronger than \cite[Theorem 4.1]{Wan18}.
\begin{theorem}\label{th:E8}
We define a modular form
$$
g:=J(A_i^{E_8}, B_j^{E_8}, \;i=1,2,3,4,5, \;j=2,3,4,6) / \Phi_{E_8} \in M_{172}(\SL_2(\ZZ)).
$$
For any $\varphi_m\in J_{k,E_8,m}^{\w,W(E_8)}$, there exists a unique polynomial $P$ over $M_*(\SL_2(\ZZ))$ in nine variables such that
$$
g^{m-1}\varphi_m=P(A_i^{E_8}, B_j^{E_8}, \;i=1,2,3,4,5, \;j=2,3,4,6).
$$
Let us define $\hat{A}_i:=A_i^{E_8}/ g^i$ and $\hat{B}_j:=B_j^{E_8}/ g^j$. Then we have
$$
J_{*,E_8,*}^{\w,W(E_8)}\subsetneq M_*(\SL_2(\ZZ))[\hat{A}_i,\hat{B}_j, \;i=1,2,3,4,5, \;j=2,3,4,6].
$$
\end{theorem}

Recall that $A_i^{E_8}$ and $B_j^{E_8}$ have Fourier expansions $1+O(q)$. By definition,  $J(A_i^{E_8},B_j^{E_8})=O(q^8)$. Thus $g=O(q^8)$, which implies that $g\in \Delta^8 \cdot M_{76}(\SL_2(\ZZ))$. Therefore $g$ is the product of $\Delta^8E_4$ and a modular form of weight $72$. It is hard to determine $g$ explicitly because these forms have unwieldy Fourier expansions. We expect that $g$ equals $\Delta^{14}E_4$ up to a constant.

It is possible to choose generators of lower weights instead of $A_i^{E_8}$ and $B_j^{E_8}$, in which case the corresponding modular form $g$ has smaller weight. By the structure results of $W(E_8)$-invariant weak Jacobi forms of small index obtained in \cite{Wan18}, the generators of the smallest possible weights should be chosen as weak Jacobi forms of weights $k_j$ and indices $m_j$ for $(k_j,m_j)=(4,1)$, $(-4,2)$, $(-2,2)$, $(-8,3)$, $(-6,3)$, $(-16,4)$, $(-14,4)$, $(-18,5)$, $(-26,6)$. If such generators exist, then the corresponding $g$ has weight 38. Thus the smallest possible weight of $g$ is 38. But we do not know if such forms of indices 5 and 6 exist. We only constructed a weak Jacobi form of weight $-16$ and index $5$ and a  weak Jacobi form of weight $-24$ and index $6$ in \cite{Wan18}. Thus the modular form $g$ that we can obtain at present has minimal weight $42$. That $g$ is the product of $E_4^3E_6$ and a modular form of weight 24 by the equality in the proof of \cite[Theorem 5.8]{Wan18}.

We do not know whether it can be inferred from Theorem \ref{th:E8} that $J_{*,E_8,*}^{\w,W(E_8)}$ is finitely generated over $M_*(\SL_2(\ZZ))$. More generally, we would like to ask the following questions.
\begin{enumerate}
\item Is $J_{*,L,*}^{\w,G}$ always finitely generated over $M_*(\SL_2(\ZZ))$?
\item Are there other lattices $L$ and subgroups $G$ of $\Orth(L)$ such that $J_{*,L,*}^{\w,G}$ is freely generated over $M_*(\SL_2(\ZZ))$?
\end{enumerate}

\bigskip

\noindent
\textbf{Acknowledgements} 
This work was inspired by a communication with Riccardo Salvati Manni. The author would like to thank him for helpful discussions. The author also thanks Valery Gritsenko for a careful reading of this paper and suggestions of improvement.
The author is grateful to Max Planck Institute for Mathematics in Bonn for its hospitality and financial support. The author also thanks the referee for his/her careful reading and useful comments.

\bibliographystyle{plainnat}

\end{document}